\documentclass[11pt,a4paper]{amsart}
\usepackage{graphicx}
\usepackage[latin1]{inputenc}
 \usepackage{float}

\usepackage{url}
\usepackage{amsthm}
\usepackage{amsthm}
\usepackage{dsfont}
\usepackage[top=7cm,bottom=5cm,left=5cm,right=5cm]{geometry}
\theoremstyle{plain} 
\newtheorem{thm}{Theorem}
\newtheorem{prop}{Proposition} 
\newtheorem{coro}{Corollary} 
\theoremstyle{definition} 

\newtheorem{lem}{Lemma}
\theoremstyle{remark} 
\newtheorem{rem}{Remark}
\newtheoremstyle %
   {sautligne}
   {1ex}
   {1ex}
   {\itshape}
   {-1ex}
   {\rmfamily\scshape}   
   {.}
   {\newline}
   {}
\usepackage{amsmath,amssymb,amssymb,graphics,amsmath}

\newcounter{exercice}

\everymath{\displaystyle}
\newcounter{probleme}

\begin{document}

\title{Exponential ergodicity for a class of non-Markovian stochastic processes}

\author{Laure Pédèches}
\address{Institut de Mathématiques de Toulouse UMR5219, Université de Toulouse,UPS IMT, F-31062 Toulouse Cedex 9, France}
\email{laure.pedeches@math.univ-toulouse.fr}
\date{June 28th, 2016}

\begin{abstract}
We prove the convergence at an exponential rate towards the invariant probability measure for a class of solutions of stochastic differential equations with finite delay. This is done, in this non-Markovian setting, using the cluster expansion method, inspired from \cite{DR} or \cite{MRZ}. As a consequence, the results hold for small perturbations of ergodic diffusions.
\end{abstract}

\maketitle

\section*{Introduction}

The aim of this paper is to prove the exponential ergodicity towards the unique invariant probability measure for some solutions of stochastic differential equations with finite delay and non regular drift.\\

We will consider $\mathbb{R}^d$-valued stochastic differential equations of the form :
\begin{equation*}
dx_t=\left(g(x_t) + \sigma \; b((x)_{t-t_0}^t)\right)dt + \sigma \; dB_t
\end{equation*}
for a small additional drift term $b$, whose only requirement is to be measurable and bounded, and with certain assumptions on the underlying semi-group of the reference process solution of \begin{equation*}
dy_t=g(y_t)dt+\sigma \; dB_t
\end{equation*}

Such an interest in those equations comes from a desire to obtain similar results for stochastic Cucker-Smale type models (such as the one presented by Ha, Lee and Levy in \cite{HLL}), which should be included in a future work.\\

While results about the existence  of invariant probability measures for stochastic differential equations with delay can be found in the literature, going back to the 1964 paper of Itô and Nisio \cite{IN} (see \cite{IKS} for a survey on the topic published in 2003 by Ivanov, Kazmerchuk and Swishchuk), they are mainly valid under strong regularity assumptions on the coefficients. Furthermore, very little has been told about exponential ergodicity for such processes.\\

As we are dealing with non-Markovian processes, most standard methods are not available to us. Our main tool here will be the so-called cluster expansion method, mainly used in statistical mechanics, in particular in Gibbs field theory. As a consequence, our results will hold for irregular but small (albeit not insignificant) perturbations of the reference process.\\

Technical results for the adaptation of the cluster expansion methods to Gibbs random fields can be found in the book \cite{MM} by Malyshev and Minlos. Subsequent papers have implemented those methods for stochastic processes, for example, interacting diffusions systems or one-dimensional non-Markovian diffusions. It was done by Ignatyuk, Malyshev and Sidoravicius in \cite{IMS}, and, more recently, amongst others in \cite{DR}, \cite{DRZ} or \cite{MRZ}.\\ 

Our main result, the exponential convergence of the semi-group towards the invariant probability measure, implies that the decay of correlations is exponentially quick, i.e. we have what is called exponential decorrelation : there exist two constants $\theta_1$ and $\theta_2$ such that for all $f$ and $g$ measurable and bounded by $1$, it holds : \begin{equation*}
|cov(f(x_t),g(x_{t+s}))| \leqslant \theta_1 \; e^{-\theta_2 s}
\end{equation*}

It follows that a central limit theorem can be derived from the mixing properties implied by this inequality.\\

Contrary to what was done in \cite{DR}, we do not require for the semi-group associated with the process to be ultracontractive, but only need some strong form of hypercontractivity. One instance of a well-known process which is not ultracontractive but verifies our assumptions is the Ornstein-Uhlenbeck process. We will present some explicit results in this particular setting. Actually, the stochastic Cucker-Smale model can be seen as a mean-field perturbation of the Ornstein-Uhlenbeck process, and this led to this work.\\

Under those less restrictive hypotheses, the lack of ultracontractive bounds for the underlying reference process introduces several new technical difficulties. We will therefore present a detailed proof for the obtaining of the cluster estimates. The use of these estimates to get to the final main theorem follows the lines of, for instance, \cite{LS} and will not be detailed here.\\

We start by introducing our framework, the objects we will manipulate and the assumptions that will be needed. Then, we set out to obtain a cluster representation for the partition function defined in the first part. In section 3, we study the cluster estimates and show that they tend to $0$ when the norm of $b$ does. We present the main results in section 4. Finally, in section 5, we explicitly compute some of the bounds in the Ornstein-Uhlenbeck setting.\\

\section{Framework and assumptions}

Let $\Omega$ be the space of real $\mathbb{R}^d$-valued continuous functions for some $d\geqslant 1$.

\subsection{A few notations}

We set the following notations :
\begin{itemize}
\item $a$, a fixed positive number, destined to become quite large ;
\item for every $j$ in $\mathbb{Z}$, $I_j=[ja, (j+1)a]$ ;
\item for every $N$ in $\mathbb{N}^*$, $I(N)=[-Na, Na]=\bigcup_{j=-N}^{N-1}I_j$ ;
\item for every $u$ in $\Omega$, $u^{(N)}(t)= u(Na)$ if $t\geqslant Na$, $u^{(N)}(t)= u(t)$ if $-Na\leqslant t \leqslant Na$, and $u^{(N)}(t)= u(-Na)$ if $t\leqslant -Na$. That is : $u^{(N)}$ is equal to $u$ frozen outside of the interval $I(N)$.
\end{itemize} 

\subsection{The stochastic processes under consideration : hypotheses (H1) and (H2)}\label{1.2}

We introduce here the process which will serve as a reference in our work : a stochastic process sufficiently regular to be exponentially ergodic towards its invariant probability measure. We present all the assumptions that will be necessary to extend this ergodicity to small perturbations of this process.\\

We consider the following stochastic differential equation :

\begin{equation}\label{1}
dy_t=g(y_t)dt+\sigma \; dB_t
\end{equation}
with $g:\mathbb{R}^d\rightarrow \mathbb{R}^d$ a smooth function, $\sigma$ an invertible $d\times d$ matrix, and $B$ a standard $d$-dimensional Brownian motion.\\

We assume that $(y_t)_{t\in\mathbb{R}_+}$ is a $\mathbb{R}^d$-valued stochastic process, solution of (\ref{1}) with initial condition $y_0=y^0$. For every $f$ regular enough, we define $P_t f$ its semi-group at time $t$, that is, for any given $x$, $P_t f(x)=\mathbb{E}_x[f(y_t)]$, and its infinitesimal generator $L=\frac{1}{2}\sum_{i,j}a_{ij}\partial^2_{ij}+\sum_i g_i \partial i$, with $a_{ij}=(\sigma\sigma^*)_{ij}$.

We will suppose that $(y_t)$ admits a symmetric probability measure $\mu$, i.e. for all $f$ and $g$ smooth enough, $\int f L g \: d\mu = \int g L f \: d\mu $. Thus, $\mu$ is an invariant measure : for every $f$ smooth enough, $\mathbb{E}_{\mu}[f(y_t)]=\int f \; d\mu$. Since there exists $\alpha>0$ such that for all $\xi\in\mathbb{R}^d$, $\xi^*\sigma\sigma^*\xi \geqslant \alpha |\xi|^2$, the generator is uniformly elliptic and it is well-known that $\mu$ is absolutely continuous with respect to Lebesgue measure with a positive density. Thus, $\mu$ is of the form $d\mu(x)=C\;e^{-V(x)}dx$ ; in addition, it is known that $V:\mathbb{R}^d\rightarrow \mathbb{R}$ is smooth.\\

According to Kolmogorov's characterization of reversible diffusions, in \cite{KK}, this even implies that $\mu$ is symmetric for the semi-group $(P_t)_{t\geqslant 0}$ and that $g$ can be written as $g=-\frac{1}{2}\;\sigma\;\nabla V$. Thus, equation (\ref{1}) becomes :
\begin{equation}\label{1r}
dy_t=-\frac{1}{2}\;\sigma\;\nabla V(y_t)dt+\sigma \; dB_t
\end{equation}

Our assumptions imply that $(y_t)$ is non-explosive and $(P_t)$ admits a smooth transition density with respect to $\mu$, denoted by $p(t,x,y)$. As the process $(y_t)_{t\geqslant 0}$ is reversible with respect to the stationary measure $\mu$, $$\forall \; t, x, y, \quad p(t,x,y)=p(t,y,x).$$

We now introduce two hypotheses which will be essential in the following :\\

\begin{itemize}
\item (H1) : $\mu$ satisfies a Poincaré inequality : there exists a constant $C_P$ such that for all smooth functions $f$ in $\mathcal{L}^2(\mu)$, 
\begin{equation*}
\|\;f-\int f\;d\mu \;\|^2_{\mathcal{L}^2(\mu)} \leqslant C_P \int |\sigma \; \nabla f |^2 \; d\mu 
\end{equation*}

\item (H2) : There exists $\delta\geqslant 0$ such that $$\sup_{t\geqslant\delta}\|p(t,., .)\|_{\mathcal{L}^8(\mu\otimes\mu)}<\infty$$
\end{itemize}

\begin{rem} Actually, we only require $\|p(t,., .)\|_{\mathcal{L}^8(\mu\otimes\mu)}$ to be finite for certain values of $t$ and it is not necessary for the supremum over $t$ to be finite. However, the current form of (H2) simplifies the writing of the proofs, and it is satisfied by the important case of the Ornstein-Uhlenbeck process, as will be seen in section \ref{OU}.
\end{rem}

\begin{rem}\label{rem2}
Using Cauchy-Schwarz's and Jensen's inequalities,

\begin{center}
$\begin{array}{lll}
\|P_{\delta} f \|_{\mathcal{L}^8(\mu)}^8 & = & \int \left( \int p(\delta,x,y)f(y)\; \mu(dy) \right)^8 \mu(dx) \\
& \leqslant & \int \left( \int |f(y)|^2 \; \mu(dy) \right)^4 \left( \int p(\delta,x,y)^2 \;\mu(dy) \right)^4 \mu(dx) \\
& \leqslant & \|f\|_{\mathcal{L}^2(\mu)}^8 \int \int p(\delta,x,y)^8 \;\mu(dy) \mu(dx) 
\end{array}$
\end{center}

This means that :
\begin{equation}\label{int1}
\|P_{\delta} f \|_{\mathcal{L}^8(\mu)} \leqslant  \|p(\delta,., .)\|_{\mathcal{L}^8(\mu\otimes\mu)} \; \|f \|_{\mathcal{L}^2(\mu)}.
\end{equation}

Thus, if (H2) is satisfied, $P_{\delta} : \mathcal{L}^2(\mu) \rightarrow \mathcal{L}^8(\mu)$ is a bounded operator and $(P_t)_t$ is hyperbounded, that is, it verifies a defective log-Sobolev inequality : there exists non-negative constants $C_{LS}$ and $\rho$ such that for all smooth functions $f$ in $\mathcal{L}^2(\mu)$, 
\begin{equation*}
\int f^2 \ln(f^2)\;d\mu - \int f^2 \;d\mu \; \ln\;(\int f^2 d\mu) \leqslant C_{LS} \int |\sigma \; \nabla f |^2 \; d\mu + \rho \int f ^2 d\mu 
\end{equation*}
If, in addition, (H1) is satisfied, the semi-group $(P_t)_t$ is hypercontractive, i.e. for every $k\geqslant2$, there exists $\delta_k$ such that for every $t\geqslant\delta_k$, $P_t:\mathbb{L}^2\rightarrow\mathbb{L}^k$ is bounded by $1$. Equivalently, it means that $\mu$ verifies a tight log-Sobolev inequality : there exists a non-negative constant $\bar{C}_{LS}$  such that for all smooth functions $f$ in $\mathcal{L}^2(\mu)$, 
\begin{equation*}
\int f^2 \ln(f^2)\;d\mu - \int f^2 \;d\mu \; \ln\;(\int f^2 d\mu) \leqslant \bar{C}_{LS} \int |\sigma \; \nabla f |^2 \; d\mu
\end{equation*}
\end{rem}

\begin{rem} It is well-known (see \cite{ABC}  for example) that hypothesis (H1) is equivalent to the exponential convergence of the semi-group towards $\mu$ :\\

(H1b) : There exists a constant $C_S$ such that for all $f\in\mathcal{L}^2(\mu)$, for all $t\geqslant 0$,
\begin{equation*}
\|\;P_t f-\int f\;d\mu \;\|_{\mathcal{L}^2(\mu)} \leqslant e^{-C_S t} \|\;f-\int f\;d\mu \;\|_{\mathcal{L}^2(\mu)} 
\end{equation*}
and, moreover, $C_S=1/C_P$ (see \cite{CGZ} for a more general statement).\\

In particular, if (H1) holds and $\int f d\mu =0$, then, for every $t\geqslant 0$,
 \begin{equation}\label{h3}
\|P_t f \|_{\mathcal{L}^2(\mu)} \leqslant e^{-t/C_P} \|f \|_{\mathcal{L}^2(\mu)} 
\end{equation}
\end{rem}

\vspace{1\baselineskip}

We now prove two lemmas, the second one, resulting in part from the first one, taking into account hypotheses (H1) and (H2) and yielding the assumption we will use in practice, rather than (H1) itself.\\

\begin{lem}\label{lem2}
Suppose that (H1) holds true. Then for all smooth $f$ such that $\int f d\mu =0$,
\begin{equation*}
\forall t \geqslant \delta, \quad \|P_t f \|_{\mathcal{L}^8(\mu)} \leqslant \|p(\delta,., .)\|_{\mathcal{L}^8(\mu\otimes\mu)} \; e^{-(t-\delta)/C_P} \; \|f \|_{\mathcal{L}^8(\mu)}  
\end{equation*}
\end{lem}

\begin{rem}
It is possible for both sides of the above inequality to be infinite.
\end{rem}

\begin{proof}
Let $t$ be a positive number with $t\geqslant\delta$ and $f$ a smooth function such that $\int f d\mu =0$.\\

For any $t\geqslant \delta$, thanks to the inequality (\ref{int1}) proven in remark \ref{rem2},

\begin{center}
$\begin{array}{lll}
\|P_{t} f \|_{\mathcal{L}^8(\mu)} & = & \|P_{\delta} (P_{t-\delta} f) \|_{\mathcal{L}^8(\mu)} \\
& \leqslant &  \|p(\delta,., .)\|_{\mathcal{L}^8(\mu\otimes\mu)} \; \|P_{t-\delta} f \|_{\mathcal{L}^2(\mu)}
\end{array}$
\end{center}

As (H1) is supposed to be satisfied, so is (\ref{h3}) ; hence the conclusion of this proof :

\begin{center}
$\begin{array}{lll}
\|P_{t} f \|_{\mathcal{L}^8(\mu)} & \leqslant &  \|p(\delta,., .)\|_{\mathcal{L}^8(\mu\otimes\mu)} \; e^{-(t-\delta)/C_P} \|f \|_{\mathcal{L}^2(\mu)} \\
 & \leqslant &  \|p(\delta,., .)\|_{\mathcal{L}^8(\mu\otimes\mu)} \; e^{-(t-\delta)/C_P} \|f \|_{\mathcal{L}^8(\mu)} 
\end{array}$
\end{center}
\end{proof}

\begin{lem}\label{lem1}
Under hypotheses (H1) and (H2), for $t\geqslant 2 \delta $, $$\| p(t,x,y)-1\|_{\mathcal{L}^8(\mu\otimes\mu)}\leqslant \beta_{\delta}(t) \left(\|p(\delta,., .)\|_{\mathcal{L}^8(\mu\otimes\mu)}\vee 1 \right)$$

where $\beta_{\delta}(t)=2 M_{\delta} \; e^{-(t-2\delta)/C_P}$ with $M_{\delta}= \sup_{a\geqslant\delta}\|p(a,., .)\|_{\mathcal{L}^8(\mu\otimes\mu)}\vee 1$
\\

In particular, $\beta_{\delta}(t)$ goes to $0$ exponentially fast when $t$ goes to infinity.\\
\end{lem}

\begin{proof}
We start by expressing $p(t,x,y)-1$ using the semi-group :\\

$\begin{array}{lll}
p(t,x,y)-1 & = & \int \left(p(t-\delta,x,z)p(\delta,z,y)-1\right)\mu(dz) \\
& = & P_{t-\delta}(p(\delta,.,y))(x)-1 \\
& = & P_{t-\delta}(p(\delta,.,y)-1)(x) 
\end{array}$

\vspace{1\baselineskip}

Thus, applying lemma \ref{lem2} for $f=p(\delta,.,y)-1$ at time $t-\delta$ ($\geqslant \delta$ as $t\geqslant 2 \delta$),\\

$\begin{array}{lll}
\int (p(t,x,y)-1)^8\mu(dx) & = & \int P^8_{t-\delta}(p(\delta,.,y)-1) (x)\;\mu(dx) \\
& \leqslant &   \|p(\delta,., .)\|_{\mathcal{L}^8(\mu\otimes\mu)}^8 \; e^{-8(t-2\delta)/C_P} \|p(\delta,.,y)-1 \|_{\mathcal{L}^8(\mu)}^8 \\
& \leqslant & 2^8\|p(\delta,., .)\|_{\mathcal{L}^8(\mu\otimes\mu)}^8 \; e^{-8(t-2\delta)/C_P} \left( \|p(\delta,., y)\|^8_{\mathcal{L}^8(\mu)}\vee 1\right),
\end{array}$\\

which leads to :

\begin{equation*}
\int (p(t,x,y)-1)^8\mu(dx) \mu(dy) \leqslant 2^8 \|p(\delta,., .)\|_{\mathcal{L}^8(\mu\otimes\mu)}^8 \; e^{-8(t-2\delta)/C_P} \left( \|p(\delta,.,.)\|^8_{\mathcal{L}^8(\mu\otimes\mu)}\vee 1 \right)
\end{equation*} 

\vspace{1\baselineskip}

Hypothesis (H2) ensures that $\|p(u,.,.)\|_{\mathcal{L}^8(\mu\otimes\mu)}$ is bounded uniformly in $u$ for $u\geqslant \delta$, hence the result.
\end{proof}

\begin{rem}
As can be seen in the proof, $M_{\delta}$ is a priori not the optimal bound (although it can correspond with $\|p(\delta,.,.)\|_{\mathcal{L}^8(\mu\otimes\mu)}$ as will be seen in the Ornstein-Uhlenbeck example in section \ref{OU}) but will be good enough for our needs (and will simplify later computations).
\end{rem}

\vspace{1\baselineskip}

\emph{In what follows, $P$ will be the law, on $\Omega$, of the stationary solution of (\ref{1r}). It will be our reference measure.}\\

\subsection{The perturbed equation}

Having established properties and assumptions needed for the reference process, we turn our attention to the stochastic differential equation with finite delay $t_0$ :

\begin{equation}\label{2}
dx_t=\left(-\frac{1}{2}\;\sigma\;\nabla V(x_t) + \sigma \; b((x)_{t-t_0}^t)\right)dt + \sigma \; dB_t
\end{equation}
where $V$, $\sigma$ and $B$ are as previously defined and $b:\mathcal{C}^0([-t_0,0],\mathbb{R}^d)\rightarrow\mathbb{R}^d$ is a bounded measurable function.\\

Let $Q$ be the law of $x$. For given numbers $a$ and $b$, and for an interval $I$, $Q_I^{a,b}$ shall be the law of the stochastic bridge associated with $x$ such that $y(\inf I)=a$ and $y(\sup I)=b$.

\begin{rem}
We only require $b$ to be bounded and measurable, without any further assumption on its regularity. This is one of the main advantages of our method.\\
We give here a few examples (that should be renormalized to be small enough to satisfy the hypotheses of our main result, Theorem \ref{th1}) :
\begin{itemize}
\item we can consider $b$ of the form $b((u)_{t-t_0}^t)=\int_{t-t_0}^t f(u_t,u_s)\; ds$ for any trajectory $u\in\Omega$ with $f$ bounded and measurable ; for instance, $f(x,y)=\text{sign} (x-y)$ or $f(x,y)=\mathds{1}_{y\in A}$ with $A$ a subset of $\mathbb{R}^d$ (thus obtaining an occupation time);
\item we may wish to have a dependence in the past depending on a single time, of the form $b((u)_{t-t_0}^t)=g(u_{t-t_0})$ for a certain function $g$ measurable and bounded, but not necessarily continuous ;
\item staying with non continuous perturbation drifts, one can also consider a function with jumps, such as $b((u)_{t-t_0}^t)=\mathds{1}_{(u)_{t-t_0}^t \in A}$ with $A$ a subset of $\mathcal{C}^0([-t_0,0],\mathbb{R}^d)$.\\
\end{itemize}
\end{rem}

Using Girsanov theorem (taken for instance from \cite{LS}), we can show that the restriction to some finite time interval $I$ of the law of $x$ is absolutely continuous with respect to the restricted law of the reference process $y$, and that its density is of the form $\exp(-H_{I}(u))du$ where the associated Hamiltonian $H_I$ is defined by

\begin{equation}
H_I(u)= - \int_I b((u)_{t-t_0}^t)^*\;dB_t + \frac{1}{2} \int_I |b((u)_{t-t_0}^t)|^2\; dt
\end{equation}

for every trajectory $u$ in the space state $\Omega$. We will denote $H_N=H_{I(N)}$.\\

Our goal is to prove the convergence of the sequence of measures $(Q_N)$, defined on $\Omega$ by \begin{equation}Q_N(du)= \exp(-H_N(u^{(N)}))P(du)\end{equation} towards a  weak solution of the equation (\ref{2}) that will be time stationary. From this point, classical results of Gibbs theory shall provide further results, such as the exponential decorrelation.\\

To this end, we introduce the partition function $Z_N$ given by \begin{equation}
Z_N=\mathbb{E}_{\mu}[\exp(-H_N)]
\end{equation}

with $\mathbb{P}_{\mu}$ being the law under $\mu$ and $\mathbb{E}_{\mu}$ the associated expectation.\\

The aim of our next section will be to write $Z_N$ under the form of a cluster expansion. What this means shall be detailed next.\\

\section{The cluster representation of the partition function}

In this paragraph, we aim to expand the partition function into a ``cluster expansion", that is to obtain an expression of $Z_N$ of the form :
$$Z_N=1+ \sum_{\tau}\prod_i \Gamma_{\tau_i}$$ with the meaning and nature of each of $\tau$, $i$ and $\Gamma_{\tau_i}$ to be determined.\\

A usual, we write $P(X= \; . \;|X'=Z)$ for the conditional distribution of $X$ given $X'=Z$, i.e. a regular desintegration of the $P$-law of $(X,X')$ with respect to the law of $X'$.\\

We start by conditioning with respect to the values taken at times $-Na$, ..., $Na$ by the process $(y_t)_t$:\\

$\begin{array}{lll}
Z_N & = & \int_{\Omega} \exp(-H_N(u^{(N)}))P(du) \\

& = & \int _{\mathbb{R}^{(2N+1)d}}\int_{\Omega} \exp(-H_N(u^{(N)}))P(du|u(ja)=y_j, j=-N,...,N) \\

&  & \quad  \quad  \quad  \quad  \quad  \quad  \quad  \quad  \bigotimes_{j=-N}^{N-1} P(y((j+1)a)=dy_{j+1}|y(ja)=y_j)\otimes P(y(-Na)=dy_{-N})\\
\end{array}$

Thus, first using the Markov property of $y$ and the definition of the bridges $Q_I^{a,b}$, and, second, introducing the transition densities $p(t,.,.)$ of $(y_t)_t$ with respect to $\mu$,\\

$\begin{array}{lll}
Z_N & = & \int _{\mathbb{R}^{(2N+1)d}}\int_{\Omega} \exp(-H_N(u^{(N)})\bigotimes_{j=-N}^{N-1} Q_{I_j}^{y_j, y_{j+1}}(du)\left(\bigotimes_{j=-N}^{N-1}P(y(a)=dy_{j+1}|y(0)=y_j)\right)\otimes \mu(dy_{-N})\\

& = & \int _{\mathbb{R}^{(2N+1)d}}\int_{\Omega} \exp(-H_N(u^{(N)}))\bigotimes_{j=-N}^{N-1} Q_{I_j}^{y_j, y_{j+1}}(du)\left(\bigotimes_{j=-N}^{N-1}p(a,y_j,y_{j+1})\mu(dy_{j+1})\right)\otimes \mu(dy_{-N})\\
\end{array}$\\

Next, we re-arrange the terms in a convenient way :

$$Z_N= \int _{\mathbb{R}^{(2N+1)d}}\int_{\Omega} \prod_{j=-N}^{N-1}\exp(-H_{I_j}(u^{(N)}))p(a,y_j,y_{j+1})\bigotimes_{j=-N}^{N-1} Q_{I_j}^{y_j, y_{j+1}}(du)\bigotimes_{j=-N}^{N}\mu(dy_{j})$$

Contrary to what was done, by mistake, between equations (13) and (14) in \cite{DR}, we cannot exchange the product and the integral over $\Omega$. This can be corrected in a way by a different decomposition : we write :
\begin{equation} Z_N=\int _{\mathbb{R}^{(2N+1)d}}\int_{\Omega} \prod_{j=-N}^{N-1}\alpha_j(a,y,u)\bigotimes_{j=-N}^{N-1} Q_{I_j}^{y_j, y_{j+1}}(du)\bigotimes_{j=-N}^{N}\mu(dy_{j})
\end{equation}
$\begin{array}{llll}
\text{  where  }\alpha_j(a,y,u)& = & \exp(-H_{I_{-N}}(u^{(N)}))\sqrt{p(a,y_{-N},y_{-N+1})} &  \text{if }j=-N\\
& = & \exp(-H_{I_j}(u^{(N)}))\sqrt{p(a,y_{j-1},y_{j})p(a,y_j,y_{j+1})} &  \text{if }j\in\{-N+1,...,N-2\}\\
& = & \exp(-H_{I_{N-1}}(u^{(N)}))\sqrt{p(a,y_{N-2},y_{N-1})}p(a,y_{N-1},y_N) & \text{if }j=N-1
\end{array}$

\vspace{1\baselineskip}

In order to obtain a sum of a product of terms that are ``temporally independent" from each other, we rewrite differently the product of the $\alpha_j$ :
\begin{center}
$\begin{array}{lll}
\prod_{j=-N}^{N-1}\alpha_j(a,y,u) & = & \prod_{j=-N}^{N-1}(1+\alpha_j(a,y,u)-1) \\

& = & 1 + \sum_S \prod_{j\in S}(\alpha_j(a,y,u)-1)\end{array}$
\end{center}

where the sum is taken on all non-empty subsets $S$ of $\{-N,...N-1\}$.\\

Thus,  $$\prod_{j=-N}^{N-1}\alpha_j(a,y,u)= 1 + \sum_{p\in\mathbb{N^*}}\sum_{\;\tau_1\sqcup...\sqcup \tau_p} \; \prod_{i=1}^p \; \prod_{j\in\tau_i}\;(\alpha_j(a,y,u)-1)$$

with $\tau_i$ of the form $\{b,b+1,...,b+r\}$, $b, b+r \in I(N)$, and $d(\tau_i,\tau_j)\geqslant 2$\\

Coming back to the expression of the partition function,\\

$\begin{array}{lll}
Z_N & = & \int _{\mathbb{R}^{(2N+1)d}}\int_{\Omega}\left( 1 + \sum_{p\in\mathbb{N^*}}\sum_{\;\tau_1\sqcup...\sqcup \tau_p \subset I(N)} \; \prod_{i=1}^p \; \prod_{j\in\tau_i}\;(\alpha_j(a,y,u)-1)\right)\bigotimes_{j=-N}^{N-1} Q_{I_j}^{y_j, y_{j+1}}(du)\bigotimes_{j=-N}^{N}\mu(dy_{j})\\

& = & 1 + \sum_{p\in\mathbb{N^*}}\sum_{\;\tau_1\sqcup...\sqcup \tau_p \subset I(N)}\int _{\mathbb{R}^{(2N+1)d}}\int_{\Omega}  \prod_{i=1}^p \; \prod_{j\in\tau_i}\;(\alpha_j(a,y,u)-1)\bigotimes_{j=-N}^{N-1} Q_{I_j}^{y_j, y_{j+1}}(du)\bigotimes_{j=-N}^{N}\mu(dy_{j})
\end{array}$\\

\vspace{1\baselineskip}

The decomposition of the product of the $\alpha_j$ was done with in mind the idea of inverting the product for $i$ from $1$ to $p$ and both integrals. This is indeed now possible :\\
\begin{itemize}
\item Notice that $j_1\in\tau_{i_1}$ and $j_2\in\tau_{i_2}$ implies that $|j_1-j_2|\geqslant 2$.

Take $\tau_i=\{b_i, ..., b_i+r_i\}$.\\

As $\prod_{j\in\tau_i}\;(\alpha_j(a,y,u)-1)$ only depends on $u(t)$ for
$$t\in \bigcup_{j\in\tau_i}[ja-t_0, (j+1)a]=[(b_i a-t_0))\wedge (-Na), (b_i +r_i +1)a] \subset  I_{b_i-1}\cup...\cup I_{b_i+r_i}$$
and $(b_{i_1}+r_{i_1}+1)a<b_{i_2}a-t_0$ for $i_1, i_2\in\{1,...,p\}$, $i_1\neq i_2$ and $a$ large enough, we have  $$\left(I_{b_{i_1}-1}\cup...\cup I_{b_{i_1}+r_{i_1}}\right)\bigcap\left(I_{b_{i_2}-1}\cup...\cup I_{b_{i_2}+r_{i_2}}\right)=\emptyset.$$

This allows us to invert the product of the $\alpha_j$ with the integral over $\Omega$ : we thus have 
\begin{equation*}
Z_N=1 + \sum_{p\in\mathbb{N^*}}\sum_{\;\tau_1\sqcup...\sqcup \tau_p} \int _{\mathbb{R}^{(2N+1)d}} \prod_{i=1}^p \left[\int_{\Omega}  \prod_{j\in\tau_i}\;(\alpha_j(a,y,u)-1)\bigotimes_{k=b_i-1}^{b_i+r_i-1} Q_{I_k}^{y_k, y_{k+1}}(du)\right]\bigotimes_{j=-N}^{N}\mu(dy_{j})
\end{equation*}

\item Moreover, notice that the expression between the square brackets only depends on $y_{b_i-1},y_{b_i},..., y_{b_i+r_i}$. As a consequence, we can interchange the integral over $\mathbb{R}^{2N+1}$ and the product in $i$.
\end{itemize}

Thus, we obtain the following cluster representation of the partition function $Z_N$ :
\begin{equation}\label{cr}
Z_N=1 + \sum_{p\in\mathbb{N^*}}\sum_{\;\tau_1\sqcup...\sqcup \tau_p \subset I(N)}\;\prod_{i=1}^p \; \Gamma_{\tau_i}
\end{equation}

where  \begin{equation}
\Gamma_{\tau}=\int _{\mathbb{R}^{(|\tau|+1)d}} \int_{\Omega}  \prod_{j\in\tau}\;(\alpha_j(a,y,u)-1)\bigotimes_{k=\min(\tau)-1}^{\max(\tau)-1} W_{I_k}^{y_k, y_{k+1}}\bigotimes_{l=\min(\tau)-1}^{\max(\tau)}\mu(dy_{l})
\end{equation}

\vspace{1\baselineskip}

and the clusters $\tau_i$ are such that :
\begin{itemize}
\item $\tau_i \subset I(N)$;
\item $j_1,j_2 \in\tau_i\; \& \; j_3 \in [j_1,j_2] \Rightarrow j_3 \in\tau_i$ (connected sets) ;
\item $j_1 \in\tau_{i_1} \; \& \; j_2 \in\tau_{i_2} \Rightarrow |j_1-j_2| \geqslant 2$ (disjoint sets).
\end{itemize}

\vspace{1\baselineskip}

\section{Study of the cluster estimates}

Having obtained the quantities $\Gamma_{\tau}$, we now wish to control them.\\

More specifically, we will show that, when the perturbation term $b$ is sufficiently small, there exists a positive function $\delta(a)$, which goes to $0$ when $a$ goes to $+\infty$, such that for $a$ large enough,
\begin{equation}\label{8}
|\Gamma_{\tau}|\leqslant\delta(a)^{|\tau|}
\end{equation}
 where $|\tau|$ is the cardinal of $\tau$.

\subsection{A generalized Hölder inequality}

In order to estimate this coefficient $\Gamma_{\tau}$, we will need to commute in some way the integrals and the remaining product (over the elements of $\tau$). The following lemma, taken from \cite{MVZ}, will be of a great help in this regard.

\begin{lem}\label{3}
Let $(\mu_x)_{x\in \mathcal{X}}$ be a family of probability measures, each one defined on a space $E_x$, where the elements $x$ belong to some finite set $\mathcal{X}$. Let us also define a finite family $(f_i)_i$ of functions on $E_{\mathcal{X}}=\times_{x\in\mathcal{X}}E_x$ such that each $f_i$ is $\mathcal{X}_i$-local for a certain $\mathcal{X}_i\subset\mathcal{X}$, in the sense that $$f_i(e)=f_i(e_{|\mathcal{X}_i}), \text{ for } e=(e_x)_{x\in \mathcal{X}}\in E_{\mathcal{X}}. $$

Let $\rho_i>0$ be numbers satisfying the following conditions : $$\forall x\in\mathcal{X}, \sum_{\mathcal{X}_i\ni x}\frac{1}{\rho_i}\leqslant 1.$$

Then $$\left| \int_{E_{\mathcal{X}}}\prod_i f_i \otimes_{x\in\mathcal{X}}d\mu_x \right| \leqslant \prod_i \left(\int_{E_{\mathcal{X}_i}} |f_i|^{\rho_i}\otimes_{x\in\mathcal{X}_i}d\mu_x\right)^{1/\rho_i}$$

\end{lem}

\vspace{1\baselineskip}

We apply lemma \ref{3} twice consecutively, first with respect to the integral over $\Omega$, then with respect to the integral over $\mathbb{R}^{(|\tau|+1)d}$. We write $\tau=\{b, ..., b+r \}$.
\begin{itemize}
\item Set $I_{\tau}(y)=\int_{\Omega} \; \prod_{j=b}^{b+r}\;(\alpha_j(a,y,u)-1)\bigotimes_{k=b-1}^{b+r-1} Q_{I_k}^{y_k, y_{k+1}}(du)$.\\

Taking $\mathcal{X}=\{b-1, ...,b+r-1\}$, $\mathcal{X}_i=\{i-1,i\}$, $E_{\mathcal{X}}=\Omega$, $E_x=\mathcal{C}^{°}(I_x,\mathbb{R^d})$ and $d\mu_x=Q_{I_x}^{y_x,y_{x+1}}$, for $(\rho_j)_{j\in\tau}$ such that $\rho_j>1$ and $\frac{1}{\rho_j}+\frac{1}{\rho_{j+1}}\leqslant 1 $, by lemma \ref{3},
$$I_{\tau}(y)\leqslant \prod_{j=b}^{b+r} g_j^{1/\rho_j}$$ where $g_j=\int_{\Omega}  |\alpha_j(a,y,u)-1|^{\rho_j}\;Q_{I_{j-1}}^{y_{j-1}, y_{j}}(du)\;Q_{I_j}^{y_j, y_{j+1}}(du)$.\\

\item Let $\overset{\sim}{\Gamma_{\tau}}=\int _{\mathbb{R}^{(r+2)d}} \; \prod_{j=b}^{b+r}\;g_j^{1/\rho_j}\bigotimes_{l=\b-1}^{b+r}\mu(dy_{l})$.\\

Choosing this time  $\mathcal{X}=\{b-1, ...,b+r\}$, $\mathcal{X}_i=\{i-1,i,i+1\}$, $E_{\mathcal{X}}=\mathbb{R}^{(r+2)d}$, $E_x=\mathbb{R}^d$ and $d\mu_x=\mu(y_x)$, for $(\gamma_j)_{j\in\{-N,...,N\}}$ such that $\gamma_j>1$ and $\frac{1}{\gamma_{j-1}}+\frac{1}{\gamma_j}+\frac{1}{\gamma_{j+1}}\leqslant 1 $, lemma \ref{3} ensures that $$ \overset{\sim}{\Gamma_{\tau}} \; \leqslant \; \prod_{j=b}^{b+r} \left( \int_{\mathbb{R}^{3d}}\; |g_j|^{\gamma_j/\rho_j}\;\mu(dy_{j-1})\;\mu(dy_{j})\;\mu(dy_{j+1})\right)^{1/\gamma_j}.$$
\end{itemize}

For every $i\in\tau$, every $j\in\{-N,...,N\}$, we take $\rho_i=\gamma_j=4$.\\

This leads us to the following inequality :
\begin{equation}
\Gamma_{\tau} \; \leqslant \;\prod_{j\in\tau} \; \left[ \int_{\mathbb{R}^{3d}}\; \int_{\Omega}\;  (\alpha_j(a,y,u)-1)^{4}\;Q_{I_{j-1}}^{y_{j-1}, y_{j}}(du)\;Q_{I_j}^{y_j, y_{j+1}}(du)\;\mu(dy_{j-1})\;\mu(dy_{j})\;\mu(dy_{j+1})\right]^{1/4}
\end{equation}

\vspace{1\baselineskip}

We now set, for every $j\in\tau$,  $$A_j(a)=\int_{\mathbb{R}^{3d}}\; \int_{\Omega}\;  (\alpha_j(a,y,u)-1)^{4}\;Q_{I_{j-1}}^{y_{j-1}, y_{j}}(du)\;Q_{I_j}^{y_j, y_{j+1}}(du)\;\mu(dy_{j-1})\;\mu(dy_{j})\;\mu(dy_{j+1}).$$

We wish to control this quantity and prove that it goes to $0$, uniformly in $j$, for a large enough $a$.\\

\subsection{Decomposition of $A_j(a)$}

Using that $$xy-1=(x-1)y+(y-1)\quad \text{ and that }\quad (xy-1)^4\leqslant 8((x-1)^4y^4+(y-1)^4)$$ for non-negative $x$ and $y$, and coming back to the expression of the $\alpha_j$, we can decompose $A_j(a)$ in two parts that will be dealt with separately :
\begin{equation}A_j(a)\leqslant 8 B_j(a) +8 C_j(a)\end{equation}

\vspace{1\baselineskip}

\begin{itemize}
\item if $j\in\{-N+1,...,N-2\}$, 
\begin{multline*}B_j(a) :=  \int_{\mathbb{R}^{3d}}\; \int_{\Omega}\;  (e^{-H_{I_j}(u^{(N)})}-1)^4\;p(a,y_{j-1},y_{j})^2\;p(a,y_j,y_{j+1})^2 \;\\
\quad \quad \quad \quad \quad \quad \quad \quad \quad \; W_{I_{j-1}}^{y_{j-1}, y_{j}}(du)\;W_{I_j}^{y_j, y_{j+1}}(du)\;\mu(dy_{j-1})\;\mu(dy_{j})\;\mu(dy_{j+1})\\
= \int_{\Omega}\;  (e^{-H_{I_j}(u^{(N)})}-1)^4 \; p(a,u((j-1)a),u(ja)) \;p(a,u(ja),u((j+1)a))P(du) 
\end{multline*}

\begin{multline*}
C_j(a)  :=  \int_{\mathbb{R}^{3d}}\int_{\Omega} \left(\sqrt{p(a,y_{j-1},y_{j})p(a,y_j,y_{j+1})}-1\right)^4W_{I_{j-1}}^{y_{j-1}, y_{j}}(du)W_{I_j}^{y_j, y_{j+1}}(du)\mu(dy_{j-1})\mu(dy_{j})\mu(dy_{j+1}) \\
=  \int_{\mathbb{R}^{3d}}\; \left(\sqrt{p(a,x,y)p(a,y,z)}-1\right)^4\mu(dx)\;\mu(dy)\;\mu(dz)
\end{multline*}

\item if $j=-N$,$$B_{-N}(a):= \int_{\Omega}\;  \left(e^{-H_{I_{-N}}(u^{(N)})}-1\right)^4 \; p(a,u(-Na),u((-N+1)a)) \;P(du)$$

$$C_{-N}(a) :=\int_{\mathbb{R}^{2d}}\; \left(\sqrt{p(a,x,y)}-1\right)^4\mu(dx)\;\mu(dy)$$\\

\item if $j=N-1$,$$B_{N-1}(a):= \int_{\Omega}\;  \left(e^{-H_{I_{N-1}}(u^{(N)})}-1\right)^4 \; p(a,u((N-2)a),u((N-1)a)) \;p(a,u((N-1)a),u(Na))^2 \;P(du)$$

$$C_{N-1}(a) :=\int_{\mathbb{R}^{3d}}\; \left(\sqrt{p(a,x,y)}p(a,y,z)-1\right)^4\mu(dx)\;\mu(dy)\;\mu(dz)$$
\end{itemize}

We will now study separately the $B_j$ and the $C_j$, without omitting the two boundary cases, especially the one when $j=N-1$, which will turn out to be the most troublesome.\\

\subsection{Study of $B_j(a)$}\label{study_bj}

\subsubsection{When $j\in\{-N+1,...,N-2\}$}

We will focus our attention on the case $j\in\{-N+1,...,N-2\}$.

Using Cauchy-Schwarz's inequality, we again decompose the integral in two parts : \\

$\begin{array}{lll}
B_j(a) & = & \int_{\Omega}\;  (e^{-H_{I_j}(u^{(N)})}-1)^4 \; p(a,u((j-1)a),u(ja)) \;p(a,u(ja),u((j+1)a))P(du) \\

& \leqslant & \left( \int_{\Omega}\;  (e^{-H_{I_j}(u^{(N)})}-1)^8 \; P(du) \right)^{1/2}\left( \int_{\Omega}\;   p(a,u((j-1)a),u(ja))^2 \;p(a,u(ja),u((j+1)a))^2P(du) \right) ^{1/2} \\
& = & K_j(a)  \left( \int_{\Omega}\;  (e^{-H_{I_j}(u^{(N)})}-1)^8 \; P(du) \right) ^{1/2}
\end{array}$\\

where $K_j(a)$ is bounded uniformly in $a$. Indeed,\\

$\begin{array}{lll}
K_j(a)^4 &  := & \left(\int_{\Omega}\;   p(a,u((j-1)a),u(ja))^2 \;p(a,u(ja),u((j+1)a))^2 \;P(du) \right)^2\\
& \leqslant &  \int_{\Omega}\;   p(a,u((j-1)a),u(ja))^4 \;P(du)\; \int_{\Omega}\; p(a,u(ja),u((j+1)a))^4 \;P(du) \\
& = & \mathbb{E}\left[p(a,y((j-1)a), y(ja))^4]  \mathbb{E}[p(a,y(ja), y((j+1)a))^4\right]\\
& = & \left(\int_{\mathbb{R}^{2d}}p(a,x,y)^4 \; p(a,x,y)\;\mu(dx) \;\mu(dy)\right)^2\\
& = & \|p(a,.,.)\|_{\mathcal{L}^5(\mu\otimes\mu)}^{10}
\end{array}$\\

\vspace{1\baselineskip}

The main goal of this subsection is to find an upper bound for $\overset{\sim}{B}_j(a)=\int_{\Omega}\;  (e^{-H_{I_j}(u^{(N)})}-1)^8 \; P(du)$ depending on $a$ going to $0$ as soon as $a$ goes to infinity.\\

What follows is a direct adaptation of what was done in \cite{RR} and \cite{DR}.\\

We start by noticing that for every $x\in\mathbb{R}$, $$(e^{-x}-1)^8=x^8\left(\int_0^1 e^{-tx}\;dt\right)^8 = x^8 \; \int_{[0,1]^8}e^{-(t_1+...+t_8)x}\;dt_1 ... dt_8 $$

and thus \begin{equation} \overset{\sim}{B}_j(a)=\int_{[0,1]^8} \int_{\Omega}\;  H_{I_j}(u^{(N)})^8e^{-(t_1+...+t_8)H_{I_j}(u^{(N)})}P(du)\;dt_1 ... dt_8
\end{equation}\\

Set $L(z)=\int_{\Omega}\; e^{-zH_{I_j}(u^{(N)})}P(du)$.\\

Then, $L$ is an holomorphic function, and its eighth derivative is $$\frac{\partial}{\partial z^8}L(z)=\int_{\Omega}\;  H_{I_j}(u^{(N)})^8e^{-zH_{I_j}(u^{(N)})}P(du)$$ which means we can rewrite $\overset{\sim}{B}_j$ as
\begin{equation}\label{4}
\overset{\sim}{B}_j(a)=\int_{[0,1]^8}\frac{\partial}{\partial z^8}L(z)\left|_{z=t_1+...+t_8} \right. \;dt_1 ... dt_8
\end{equation}

On another side, we can obtain an alternative expression of $L$, using Cauchy-Schwarz's inequality and the martingale property of $\exp(-H)$ :

\begin{multline*}
L(z)= \int_{\Omega} \; \exp \left(z\int_{I_j} b((u)_{t-t_0}^t)^*\;dB_t - \frac{z}{2} \int_{I_j} |b((u)_{t-t_0}^t)|^2 \; dt \right) \; P(du)  \\
= \int_{\Omega} \; \exp\left(z\int_{ja}^{(j+1)a} \; b((u)_{t-t_0}^t)^*\;dB_t-z^2\int_{ja}^{(j+1)a} \;|b((u)_{t-t_0}^t)|^2 \; dt\right) \\
\times \; \exp\left(\frac{2z^2-z}{2}\int_{ja}^{(j+1)a} \; |b((u)_{t-t_0}^t)|^2 \; dt\right)\; P(du) \\
\leqslant \left( \int_{\Omega} \; \exp\left(2z\int_{ja}^{(j+1)a} \; b((u)_{t-t_0}^t)^*\; dB_t-2z^2\int_{ja}^{(j+1)a} \; |b((u)_{t-t_0}^t)|^2 \; dt\right)P(du) \right)^{1/2} \\
\times \left( \int_{\Omega} \; \exp\left(z(2z-1)\int_{ja}^{(j+1)a} \; |b((u)_{t-t_0}^t)|^2 \; dt\right) P(du)\right)^{1/2}\\
= \left( \int_{\Omega} \; \exp\left(z(2z-1)\int_{ja}^{(j+1)a} \; |b((u)_{t-t_0}^t)|^2 \; dt\right) P(du)\right)^{1/2}
\end{multline*}

Then, we apply Cauchy's inequality to $L$, for $\rho$ such that $L$ is well defined on $B(z,\rho)$ : \begin{equation}\label{5}
\left| \frac{\partial}{\partial z^8}L(z)\right|\leqslant \frac{8!}{\rho^8} \; \sup_{v\in B(z,\rho)}|L(v)|
\end{equation}

Thanks to the above expression of $L$, $$|L(v)|^2 \leqslant \int_{\Omega} \; \exp\left(\mathfrak{Re}(v(2v-1))\int_{ja}^{(j+1)a} \; |b((u)_{t-t_0}^t)|^2 \; dt\right) P(du)$$

Furthermore, $\mathfrak{Re}(v(2v-1))=2\mathfrak{Re}(v)^2-2Im(v)^2-\mathfrak{Re}(v)\leqslant \mathfrak{Re}(v)(2\mathfrak{Re}(v)-1)$ and as $|v-z|^2=\rho^2$, we have $(\mathfrak{Re}(v)-z)\leqslant \rho^2$ for $z=t_1+...+t_8$, it follows that $\mathfrak{Re}(v)\in \{x\in\mathbb{R} : |z-x|<\rho \}$, hence $\mathfrak{Re}(v)\leqslant z+\rho$, which implies $\mathfrak{Re}(v(2v-1))\leqslant (z+\rho)(2z+2\rho-1) \leqslant 2(z+\rho)^2$.\\

Subsequently, $|L(v)|^2 \leqslant \int_{\Omega} \; \exp\left(2(\rho+z)^2)a \|b\|_{\infty}^2\right) P(du)$,  and thus,
\begin{equation}\label{6}
\sup_{v\in\mathcal{C}(z,\rho)}|L(v)| \leqslant \exp\left(a\|b\|_{\infty}^2(\rho+z)^2)\right)
\end{equation}

\vspace{1\baselineskip}

Combining (\ref{4}), (\ref{5}) and (\ref{6}),
$$\overset{\sim}{B}_j(a)\leqslant \int_{[0,1]^8} \frac{8!}{\rho^8}\exp\left(a\|b\|_{\infty}^2(\rho+t_1+...+t_8)^2\right)\; dt_1...dt_8 \leqslant \frac{8!}{\rho^8}\exp\left(a\|b\|_{\infty}^2(\rho+8)^2)\right)$$

Thus, for every $\rho\geqslant 8$, \begin{equation}\label{7}
\overset{\sim}{B}_j(a)\leqslant\frac{8!}{\rho^8}e^{4a\|b\|_{\infty}^2\rho^2}
\end{equation}

We want to determine which $\rho\geqslant8$ will minimize the right hand side of this last inequality.\\

Let $f$ be the function given by $f(\rho)=\frac{8!}{\rho^8}e^{4a\|b\|_{\infty}^2\rho^2}$. Then, $f^{'}(\rho)=\left(-\frac{8}{\rho}+8a\|b\|_{\infty}^2\rho\right)f(\rho)$.\\

Thus, $f^{'}(\rho)=0$ if and only if $\rho^2=\frac{1}{a\|b\|_{\infty}^2}$, which is larger than $8$ if and only if  \begin{equation}\label{c}
a\|b\|_{\infty}^2\leqslant \frac{1}{8},
\end{equation} and the optimal inequality for (\ref{7}) is \begin{equation}\label{9}
\overset{\sim}{B}_j(a)\leqslant8!e^4(a\|b\|_{\infty}^2)^4
\end{equation}

Finally coming back to the expression of $B_j$, we have obtained, under condition (\ref{c}),
\begin{equation}\label{17}
B_j(a)\leqslant \sqrt{8!}e^2\;\|p(a,.,.)\|_{\mathcal{L}^5(\mu\otimes\mu)}^{5/2}\;(a\|b\|_{\infty}^2)^2 
\end{equation}

\subsubsection{Boundary cases}

Remember that $$B_{-N}(a)= \int_{\Omega}\;  \left(e^{-H_{I_{-N}}(u^{(N)})}-1\right)^4 \; p(a,u(-Na),u((-N+1)a)) \;P(du).$$

As in the previous case, we can write 
$$B_{-N}(a) \leqslant K_{-N}(a)  \left( \int_{\Omega}\;  (e^{-H_{I_{-N}}(u^{(N)})}-1)^8 \; P(du) \right) ^{1/2}$$

where  $K_{-N}(a)^2=\int_{\Omega}\;  p(a,u(-Na),u((-N+1)a))^2 \;P(du)$\\

This square root can be dealt with in exactly the same fashion as is done above.\\

Furthermore, 
\begin{center}
$\begin{array}{lll}
K_{-N}(a)^2 & = & \mathbb{E}[p(a,y(-Na),y((-N+1)a))^2]\\
& = & \int_{\mathbb{R}^{2d}} p(a,x,y)^2p(a,x,y)\;\mu(dx)\;\mu(dy)\\
& = & \|p(a,., .)\|^3_{\mathcal{L}^3(\mu\otimes\mu)}\\
\end{array}$
\end{center}

Hence the following result : \begin{equation}\label{118}
B_{-N}(a) \leqslant \sqrt{8!}e^2\;\|p(a,.,.)\|_{\mathcal{L}^3(\mu\otimes\mu)}^{3/2}\;(a\|b\|_{\infty}^2)^2 
\end{equation}

\vspace{1\baselineskip}

We now turn our attention to $$B_{N-1}(a)= \int_{\Omega}\;  \left(e^{-H_{I_{N-1}}(u^{(N)})}-1\right)^4 \; p(a,u((N-2)a),u((N-1)a)) \;p(a,u((N-1)a),u(Na))^2 \;P(du)$$

We proceed in a similar way to decompose $B_{N-1}(a)$ into the product of two terms and we have to study the quantity :
$$K_{N-1}(a)=\sqrt{\int_{\Omega}\;  p(a,u((N-2)a),u((N-1)a))^2 \;p(a,u((N-1)a),u(Na))^4  \;P(du)}$$

In order to obtain an upper bound a moment of $p(a,.,.)$ with respect to $\mu \otimes \mu$ smaller than 8, Cauchy-Schwarz's inequality will not suffice : we have to apply Hölder's inequality. We choose the conjugated numbers $3$ and $3/2$ :$$K_{N-1}(a)^2 \leqslant \left( \int_{\Omega}  p(a,u((N-2)a),u((N-1)a))^6P(du) \right)^{1/3}  \left(\int_{\Omega}  \;p(a,u((N-1)a),u(Na))^6 P(du)\right)^{2/3}$$

which leads to \begin{equation*}
K_{N-1}(a) \leqslant \|p(a,.,.)\|_{\mathcal{L}^7(\mu\otimes\mu)}^{7}
\end{equation*}

and subsequently to 
\begin{equation}\label{19}
B_{N-1}(a) \leqslant \sqrt{8!}e^2\; \|p(a,.,.)\|_{\mathcal{L}^7(\mu\otimes\mu)}^{7/2}\;(a\|b\|_{\infty}^2)^2 
\end{equation}

\subsubsection{An uniform bound for $B_j(a)$}

From (\ref{17}), (\ref{118}) and (\ref{19}), we deduce that, for every $j\in\{-N,..., N-1\}$, \begin{equation}\label{20}
B_{j}(a) \leqslant  \sqrt{8!}e^2\; \|p(a,.,.)\|_{\mathcal{L}^7(\mu\otimes\mu)}^{7/2}\;(a\|b\|_{\infty}^2)^2 
\end{equation}

\vspace{1\baselineskip}

\subsection{Study of $C_j(a)$}\label{study_cj}

\subsubsection{General cases}

As in the previous paragraph, we will first focus on the more general situation, when $j\in\{-N+1, ... , N-2 \}$. \\

We remind that 
$$C_j(a)  =  \int_{\mathbb{R}^{3d}}\; \left(\sqrt{p(a,x,y)p(a,y,z)}-1\right)^4\mu(dx)\;\mu(dy)\;\mu(dz)$$\\

Again,  we seek an upper bound for $C_j(a)$ which vanishes when $a$ goes to infinity.\\

It can be easily checked that for every positive $U$, \begin{equation}\label{18}
(\sqrt{1+U}-1)^4\leqslant \frac{1}{16}U^4,
\end{equation}
that for positive $x$ and $y$, \begin{equation}
xy-1=(x-1)(y-1)+(x-1)+(y-1)
\end{equation}
and that, thanks to the convexity of $u\mapsto u^4$, for any $a$, $b$ and $c$,\begin{equation}
(a+b+c)^4\leqslant 27 \; (a^4+b^4+c^4)
\end{equation}
Subsequently,

$\begin{array}{lll}
C_j(a) & \leqslant & \frac{1}{16} \int_{\mathbb{R}^{3d}}\; \left(p(a,x,y)p(a,y,z)-1\right)^4\mu(dx)\;\mu(dy)\;\mu(dz) \\
& \leqslant & \frac{1}{16} \int_{\mathbb{R}^{3d}}\;\left( (p(a,x,y)-1)(p(a,y,z)-1)+(p(a,x,y)-1)+(p(a,y,z)-1)\right)^4\mu(dx)\;\mu(dy)\;\mu(dz) \\
& \leqslant & \frac{27}{16}\int_{\mathbb{R}^{3d}}\;\left( (p(a,x,y)-1)(p(a,y,z)-1)\right)^4\mu(dx)\;\mu(dy)\;\mu(dz)+ \frac{27}{8}\int_{\mathbb{R}^{2d}}\; (p(a,x,y)-1)^4\mu(dx)\;\mu(dy)\\
& \leqslant & \frac{27}{16}\int_{\mathbb{R}^{2d}}\;(p(a,x,y)-1)^8\mu(dx)\;\mu(dy)\;+ \frac{27}{8}\int_{\mathbb{R}^{2d}}\; (p(a,x,y)-1)^4\mu(dx)\;\mu(dy)\\
\end{array}$\\

using, once more, Cauchy-Schwarz's inequality to obtain the final line.\\

Furthermore,
\begin{equation*}
\int_{\mathbb{R}^{2d}}\; (p(a,x,y)-1)^4\mu(dx)\;\mu(dy)  \leqslant \sqrt{\int_{\mathbb{R}^{2d}}\; (p(a,x,y)-1)^8\mu(dx)\;\mu(dy)}
\end{equation*}

Thus, according to lemma \ref{lem1}, 
\begin{equation}
C_j(a) \leqslant \frac{27}{16}\beta_{\delta}(a)^8 \left(\|p(\delta,., .)\|^8_{\mathcal{L}^8(\mu\otimes\mu)}\vee 1 \right) + \frac{27}{8} \beta_{\delta}(a)^4 \left(\|p(\delta,., .)\|^4_{\mathcal{L}^8(\mu\otimes\mu)}\vee 1 \right)
\end{equation}

\subsubsection{Boundary cases}
We can check that both boundary cases exhibit an analogous behaviour.\\

Indeed, on the one hand, recall that $$C_{-N}(a) =\int_{\mathbb{R}^{2d}}\; \left(\sqrt{p(a,x,y)}-1\right)^4\mu(dx)\;\mu(dy)$$

Thus, thanks to (\ref{18}) and lemma \ref{lem1},
\begin{center}
$\begin{array}{lll}
C_{-N}(a) & \leqslant & \frac{1}{16}\int_{\mathbb{R}^{2d}}\; \left(p(a,x,y)-1\right)^4\mu(dx)\;\mu(dy)\\
& \leqslant & \frac{1}{16} \sqrt{\int_{\mathbb{R}^{2d}}\; \left(p(a,x,y)-1\right)^8\mu(dx)\;\mu(dy)}\\
& \leqslant & \frac{1}{16}\beta_{\delta}(a)^4 \left(\|p(\delta,., .)\|^4_{\mathcal{L}^8(\mu\otimes\mu)}\vee 1 \right)\\
\end{array}$
\end{center}

On the other hand, $$C_{N-1}(a) =\int_{\mathbb{R}^{3d}}\; \left(\sqrt{p(a,x,y)}p(a,y,z)-1\right)^4\mu(dx)\;\mu(dy)\;\mu(dz)$$

Notice that $\quad (\sqrt{p(a,x,y)}p(a,y,z)-1)^4 \leqslant 8(p(a,y,z)-1)^4 p(a,x,y)^2+ 8(\sqrt{p(a,x,y)}-1)^4$.\\

This, with Cauchy-Schwarz's inequality and the computation of $C_{-N}(a)$ thrown in, leads to
\begin{center}
$\begin{array}{lll}
C_{N-1}(a) & \leqslant & 8\int_{\mathbb{R}^{3d}}(p(a,y,z)-1)^4 p(a,x,y)^2\mu(dx)\;\mu(dy)\;\mu(dz) \\
& & \quad \quad \quad\quad \quad \quad\quad \quad \quad + \; 8\int_{\mathbb{R}^{2d}}(\sqrt{p(a,x,y)}-1)^4\mu(dx)\;\mu(dy)\\
& \leqslant & 8 \;\sqrt{\int_{\mathbb{R}^{2d}}(p(a,x,y)-1)^8\mu(dx)\;\mu(dy)}\;\sqrt{\int_{\mathbb{R}^{2d}}p(a,x,y)^4\mu(dx)\;\mu(dy)} \; \\
& &  \quad \quad \quad\quad \quad \quad\quad \quad \quad + \;\frac{1}{2}\beta_{\delta}(a)^4 \left(\|p(\delta,., .)\|^4_{\mathcal{L}^8(\mu\otimes\mu)}\vee 1 \right) \\
& \leqslant & 8 \beta_{\delta}(a)^4 \left(\|p(\delta,., .)\|^4_{\mathcal{L}^8(\mu\otimes\mu)}\vee 1 \right) \|p(a,., .)\|^2_{\mathcal{L}^4(\mu\otimes\mu)} \\
& &  \quad \quad \quad\quad \quad \quad\quad \quad \quad + \;\frac{1}{2}\beta_{\delta}(a)^4 \left(\|p(\delta,., .)\|^4_{\mathcal{L}^8(\mu\otimes\mu)}\vee 1 \right) \\
& = & \left( 8 \|p(a,., .)\|^2_{\mathcal{L}^4(\mu\otimes\mu)} + \frac{1}{2}\right)\beta_{\delta}(a)^4 \left(\|p(\delta,., .)\|^4_{\mathcal{L}^8(\mu\otimes\mu)}\vee 1 \right)\\
\end{array}$
\end{center}

\subsubsection{Global upper bound for $C_j$}

Taking into account all three cases, when $a\geqslant 2 \delta$, for every $j\in\{-N,..., N-1\}$, 
\begin{equation}\label{21}
C_j(a)\leqslant \left(\left(\frac{27}{16}\beta_{\delta}(a)^4 +8 \right)\left(\|p(\delta,., .)\|^4_{\mathcal{L}^8(\mu\otimes\mu)}\vee 1 \right)+4\right)\beta_{\delta}(a)^4 \left(\|p(\delta,., .)\|^4_{\mathcal{L}^8(\mu\otimes\mu)}\vee 1 \right)
\end{equation}

\vspace{1\baselineskip}

To obtain the control of $|\Gamma_{\tau}|$ we are seeking, it remains to put all the pieces together and to determine its domain of validity with respect to $a$ and $b$.\\

\subsection{Back to the clusters}

Suppose that $a\geqslant2 \delta$ and that (\ref{c}) holds, i.e. $\|b\|_{\infty}\leqslant \frac{1}{\sqrt{8a}}$.

Recall that \begin{equation}\label{M}
M_{\delta}= \sup_{a\geqslant\delta}\|p(a,., .)\|_{\mathcal{L}^8(\mu\otimes\mu)}\vee 1 
\end{equation}

Thus, we can obtain bounds for $B_j$ and $C_j$ easier to deal with : according to (\ref{20}) and (\ref{21}) respectively, 

\begin{equation}\label{30}
B_{j}(a) \leqslant  \sqrt{8!}e^2\;M_{\delta}^{7/2}\;(a\|b\|_{\infty}^2)^2 
\end{equation}

\begin{equation}
C_j(a)\leqslant  M_{\delta}^4\left(4+2M_{\delta}^4\left(4+\beta_{\delta}(a)^4\right)\right)\beta_{\delta}(a)^4
\end{equation}

\vspace{1\baselineskip}

We aim to prove that, for every positive $\varepsilon$, for $a$ large enough (i.e. $a$ larger than some $a_{\varepsilon}$), there exists $b(\varepsilon)$ such that if $\|b\|_{\infty}\leqslant b(\varepsilon)$, \begin{equation}\label{cc}
|\Gamma_{\tau}| \leqslant \varepsilon ^{|\tau|}.
\end{equation}

Remember that \begin{equation*}
|\Gamma_{\tau}(a)| \leqslant \;\prod_{j\in\tau}\left(8B_j(a)+8C_j(a)\right)^{1/4},
\end{equation*}

so (\ref{cc}) will be satisfied if, for $a$ sufficiently large, both $B_j(a)$ and $C_j(a)$ are smaller than $\frac{\varepsilon^4}{16}$.\\

\begin{itemize}
\item One can check, by solving a second order inequality in $\beta_{\delta}(a)^4$, that for all $a$ such that\begin{equation}\label{cj}
\beta_{\delta}(a)^4 \leqslant \left( 2 +\frac{1}{M_{\delta}^4}\right) \left( \sqrt{1+\frac{\varepsilon^4 }{32(1+2M_{\delta}^4)^2}}-1 \right)
\end{equation}
the condition $C_j(a)\leqslant \frac{\varepsilon^4}{16}$ is true.\\


We remind that $\beta_{\delta}$ was introduced in lemma \ref{lem1} and is defined by $\beta_{\delta}(a)=2 M_{\delta} \; e^{-(a-2\delta)/C_P}$, with $C_P$ the constant associated with the Poincaré's inequality satisfied by $\mu$, according to hypothesis (H1).\\

Using this expression, and setting\begin{equation}\label{cj3}
a_C(\varepsilon) = 2 \delta - \frac{C_P}{4} \ln \left( \frac{1}{16 M_{\delta}^4}\left( 2 +\frac{1}{M_{\delta}^4}\right) \left( \sqrt{1+\frac{\varepsilon^4 }{32(1+2M_{\delta}^4)^2}}-1 \right) \right)
\end{equation}
it can be shown that (\ref{cj}) is equivalent to :
\begin{equation*}\label{cj2}
a \geqslant a_C(\varepsilon) 
\end{equation*}

Thus, for every $a\geqslant a_C (\varepsilon)$, $C_j(a)\leqslant \frac{\varepsilon^4}{16}$.\\

\begin{rem}\label{rem_a}
It can be noticed that $ a_C (\varepsilon) \geqslant 2 \delta$ if and only if $\varepsilon \leqslant 2^{5/2}M_{\delta}^2(8M_{\delta}^8+2M_{\delta}^4+1)^{1/4}$.
\end{rem}

\item From (\ref{30}), it can be seen that $B_j(a)\leqslant \frac{\varepsilon^4}{16}$ if \begin{equation}\label{33}
\|b\|_{\infty} \leqslant \frac{1}{2\sqrt{e}(8!)^{1/8}M_{\delta}^{7/8}}\;\frac{\varepsilon}{\sqrt{a}} 
\end{equation}

\begin{rem}\label{rem_b}
Notice that $\frac{\varepsilon}{2\sqrt{e}(8!)^{1/8}M_{\delta}^{7/8}} \leqslant \frac{1}{\sqrt{8}}$ if and only if $\varepsilon \leqslant \sqrt{\frac{e}{2}}(8!)^{1/8}M_{\delta}^{7/8}$.
\end{rem}
\end{itemize}
\vspace{1\baselineskip}

Thus, according to (\ref{c}), (\ref{cj}) and (\ref{33}), setting
\begin{equation*}
a_{\varepsilon}= a_C(\varepsilon) \vee (2 \delta)
\end{equation*}

that is 
\begin{equation}\label{37}
a_{\varepsilon}=2 \delta - \left[\frac{C_P}{4} \ln \left( \frac{1}{16 M_{\delta}^4}\left( 2 +\frac{1}{M_{\delta}^4}\right) \left( \sqrt{1+\frac{\varepsilon^4 }{32(1+2M_{\delta}^4)^2}}-1 \right) \right)\right]_-
\end{equation}

where $x_-=min (x,0)$, and 
\begin{equation*}b(\varepsilon)=\left(\frac{\varepsilon}{2\sqrt{e}(8!)^{1/8}M_{\delta}^{7/8}} \; \wedge \; \frac{1}{\sqrt{8}} \right) \;\frac{1}{\sqrt{a_{\varepsilon}}}, 
\end{equation*}
that is,
\begin{equation}\label{38}b(\varepsilon)=\frac{\frac{\varepsilon}{2\sqrt{e}(8!)^{1/8}M_{\delta}^{7/8}} \; \wedge \; \frac{1}{\sqrt{8}} }{\sqrt{2 \delta - \left[\frac{C_P}{4} \ln \left( \frac{1}{16 M_{\delta}^4}\left( 2 +\frac{1}{M_{\delta}^4}\right) \left( \sqrt{1+\frac{\varepsilon^4 }{32(1+2M_{\delta}^4)^2}}-1 \right) \right)\right]_-} }
\end{equation}

we have finally obtained the following result :

\begin{prop}\label{prop1}
Assume that (H1) and (H2) are satisfied. Suppose that $a\geqslant a_{\varepsilon}$, defined in (\ref{37}).\\
Let $\varepsilon$ be a positive number such that $\|b\|_{\infty} \leqslant b(\varepsilon)$, given by (\ref{38}).\\
Then, for every cluster $\Gamma_{\tau}$,
\begin{equation}
|\Gamma_{\tau}| \leqslant \varepsilon ^{|\tau|}.
\end{equation}
\end{prop}

\vspace{1\baselineskip}

\begin{rem} In order to connect with classical results about the cluster expansion method, we have to show that $\varepsilon$ can be expressed as a function of $b(\varepsilon)$ that will go to $0$ when $b(\varepsilon)$ goes to 0.\\

Suppose that \begin{equation}\label{epsilon_0}
\varepsilon \leqslant \left(2^{5/2}M_{\delta}^2(8M_{\delta}^8+2M_{\delta}^4+1)^{1/4}\right) \; \wedge \; \left(\sqrt{\frac{e}{2}}(8!)^{1/8}M_{\delta}^{7/8}\right) := \varepsilon_0. 
\end{equation} 

Then according to remarks \ref{rem_a} and \ref{rem_b}, 
\begin{equation}\label{b_epsilon}b(\varepsilon)=\frac{\varepsilon} {2\sqrt{2e\delta\;}(8!)^{1/8}M_{\delta}^{7/8}\sqrt{1 - \frac{C_P}{8 \delta} \ln \left( \frac{1}{16 M_{\delta}^4}\left( 2 +\frac{1}{M_{\delta}^4}\right) \left( \sqrt{1+\frac{\varepsilon^4 }{32(1+2M_{\delta}^4)^2}}-1 \right) \right)} }
\end{equation}

Compute the derivative of $b(\varepsilon)$ with respect to $\varepsilon$ :
\begin{equation*}b^{'}(\varepsilon)=\frac{b(\varepsilon)}{\varepsilon}\left(1+\frac{(8!)^{1/4}e \;C_P \; M_{\delta}^{7/4}}{32\; \delta (1+2M_{\delta}^4)^2} \frac{\varepsilon^2 \; b(\varepsilon)^2}{\left(\sqrt{1+\frac{\varepsilon^4 }{32(1+2M_{\delta}^4)^2}}-1\right)\sqrt{1+\frac{\varepsilon^4 }{32(1+2M_{\delta}^4)^2}}}\right)
\end{equation*}
$b^{'}(\varepsilon)$ is positive for every $\varepsilon$ in $(0, \varepsilon_0]$ ; thus, $\varepsilon \mapsto b(\varepsilon)$ is (strictly) increasing from $(0, \varepsilon_0]$ to $(0, b_{\varepsilon_0}]$.\\

Therefore, $b$ admits a reciprocal function : there exists a function $\eta : (0, b_{\varepsilon_0}] \rightarrow (0, \varepsilon_0]$ such that for every $\varepsilon \in (0, \varepsilon_0]$, $\eta (b(\varepsilon))=\varepsilon$ ; moreover, $\eta$ is increasing and $\eta(x)$ goes to $0$ when $x$ goes to $0$.
\end{rem}

We can now rewrite proposition \ref{prop1} in a more amenable way :
\begin{prop}\label{prop2}
Let $\varepsilon_0$ be as defined in (\ref{epsilon_0}) and associate it $b_{\varepsilon_0}$ according to (\ref{b_epsilon}).\\
Suppose that $a\geqslant a_{\varepsilon_0}$, as introduced in (\ref{37}).\\

If $\|b\|_{\infty} \leqslant b_{\varepsilon_0}$, then, for every cluster $\Gamma_{\tau}$,
\begin{equation}\label{ce}
|\Gamma_{\tau}| \leqslant \eta(\|b\|_{\infty})^{|\tau|}
\end{equation}
where $\eta$, defined just above, is a function that goes to $0$ in $0$.\\
\end{prop}

\section{Main theorem and consequences}

Recall that we wish to prove the convergence of the sequence of measures $(Q_N)_N$, with $Q_N(du)= \exp(-H_N(u^{(N)}))P(du)$, towards a  weak stationary solution $Q$ of the perturbed equation (\ref{2}).\\

Proposition \ref{prop2}, just above, is the key point to prove this convergence : the cluster representation (\ref{cr}) of the partition function $Z_N$ and the cluster estimate (\ref{ce}) are the crucial elements in order to obtain in a canonical way an expansion for the measures $Q_N$ (see \cite{MM}).\\

It has been done in details in both \cite{DR} and \cite{MRZ} (see for instance paragraph 3.1.4 of \cite{MRZ}, with lemma 10 and what follows) : one obtains a representation of a localized bounded function with respect to the measure $Q_N$ and proves its convergence, uniformly in $N$ when $\varepsilon$ is small enough.\\

This leads to the following result : 

\begin{thm} Assume (H1) and (H2). For $\|b\|_{\infty}$ small enough, there exists a unique probability measure $Q$ on $\Omega$ such that : 
\begin{equation*}
Q=\lim_{N\rightarrow\infty} Q_N
\end{equation*}
\end{thm} 

One should note that despite the explicit bounds obtained in (\ref{38}) and (\ref{epsilon_0}), the cluster expansion method does not furnish an explicit expression for the required smallness of $\|b\|_{\infty}$.\\

Further results taken from Gibbs field theory (see proposition 2 and lemma 4 in \cite{DR}) ensure that the probability measure $Q$ is indeed a weak stationary solution of the equation :
\begin{equation}\tag{5}
dx_t=\left(-\frac{1}{2}\;\sigma\;\nabla V(x_t) + \sigma \; b((x)_{t-t_0}^t)\right)dt + \sigma \; dB_t
\end{equation}

Hence our main result :

\begin{thm}\label{th1}

Assume (H1) and (H2).

If $\|b\|_{\infty}$  is small enough, then

\begin{itemize}
\item the semi-group associated with the above equation (\ref{2}) converges towards its unique invariant probability measure $\nu$ exponentially fast ; 

\item furthermore, there exists a constructive way of obtaining this measure $\nu$ ;

\item finally, the property of exponential decorrelation is satisfied : there exist two constants $\theta_1$ and $\theta_2$ such that for all $f$ and $g$ measurable and bounded by $1$, \begin{equation*}
|cov(f(x_t),g(x_{t+s}))| \leqslant \theta_1 \; e^{-\theta_2 s}
\end{equation*}
\end{itemize}

\end{thm}

The final assertion follows from the fact that the short-range correlations hold (see, for instance Theorem 3.1 and Corolllary 3.1 in \cite{RS}).\\

As the correlations decay at an exponential rate, we have strong mixing properties, and, in particular, the central limit theorem below. Though this process is not Markovian, the proof of the following corollary is similar to the famous result obtained by Kipnis and Varadhan in \cite{KV} expanded in \cite{CCG}.

\begin{coro}
If a smooth $f$ is such that $\int f d\nu=0$, then
\begin{equation*}
\frac{1}{\sqrt{t}} \int_0^t f(x_s) \; ds \overset{(d)}{\underset{t\rightarrow +\infty}{\longrightarrow}} \mathcal{N}(0,\sigma^2_f)
\end{equation*}
with $$\sigma^2_f:=2\int_{-\infty}^{+\infty} \mathbb{E}_{\nu}[f(x_0)f(x_s)] \;ds= \int |\sigma \nabla f |^2 \; d\nu$$
\end{coro}

\begin{rem}
Similar results hold if the diffusion matrix $\sigma$ is not square, under the assumption that $\sigma \sigma ^*$ is invertible.
\end{rem}

\begin{rem}
We recall hypothesis (H2) : there exists $\delta\geqslant 0$ such that $$\sup_{t\geqslant\delta}\|p(t,., .)\|_{\mathcal{L}^8(\mu\otimes\mu)}<\infty.$$

It is not optimal, in the sense that it could be weakened and our results would still hold, with similar computations. Indeed, if, instead of Cauchy-Schwarz's inequality, we would use Hölder's inequality for the study of $B_j(a)$ (resp. $C_j(a)$) in section \ref{study_bj} (resp. \ref{study_cj}), we could substitute the power $8$ in (H2) by $4+\eta$ for any $\eta>0$ (resp. by $6$), the limiting case being $B_{N-1}$ (resp. $C_{N-1}$).\\

Thus, (H2) could be replaced by (H2') :  there exists $\delta\geqslant 0$ such that $$\sup_{t\geqslant\delta}\|p(t,., .)\|_{\mathcal{L}^6(\mu\otimes\mu)}<\infty,$$

the counterpart being that condition (\ref{c}) would then be more restrictive and lead to a greater $a_{\varepsilon}$ for a given $b(\varepsilon)$.
\end{rem}

\section{A concrete example : the Ornstein-Uhlenbeck case}\label{OU}

Suppose the reference drift $g$ is a linear one.\\

In order to simplify the writing of the computations, we restrict ourselves to the one-dimensional situation $d=1$ ; the behaviour in higher dimensions is completely similar. \\

We are thus considering the one-dimensional Ornstein-Uhlenbeck equation :
\begin{equation}
dy_t=-\lambda y_t \; dt + \sigma \; dB_t
\end{equation}
where $\lambda$ and $\sigma$ are positive parameters and $(B_t)$ is a standard one-dimensional Brownian motion.\\

It is a process whose explicit expression (with respect to the Brownian motion $B$) and general behaviour are well-known ; in particular, it admits the Gaussian law $\mathcal{N}(0,\sigma^2/2\lambda)$, whose density is given by \begin{equation*}\mu(dy)=\sqrt{\frac{\lambda}{\pi \sigma^2}}e^{-\lambda y^2 /\sigma^2}dy,
\end{equation*} as its (unique) symmetric probability measure.\\

Furthermore, the transition density of $(y_t)$ with respect to $\mu$ is given by 
\begin{equation}
p(t,x,y)=\frac{1}{\sqrt{1-e^{-2\lambda t}}}\exp\left(-\frac{\lambda}{\sigma^2(1-e^{-2\lambda t})}\left((x^2+y^2)e^{-2\lambda t}-2xye^{-\lambda t}\right)\right).
\end{equation}

Thus, all the assumptions made at the beginning of section \ref{1.2} are satisfied, as are hypotheses (H1) and (H2). Indeed, thanks to a well-known result (see, for instance, \cite{ABC}) on Poincaré's inequalities verified by Gaussian measures, for a smooth function $f$,
\begin{equation}
Var_{\mu}(f)\leqslant \frac{\sigma^2}{2\lambda} \int  \; (f')^2 \; d\mu 
\end{equation}
which implies that (H1) holds, with $C_P=1/2\lambda$.\\

Moreover, (H2) follows from the lemma below :

\begin{lem}
For every positive $t$ and for $k \in \mathbb{N}^*$,

\begin{equation}\label{ll}
\int_{\mathbb{R}^2} p(t,x,y)^k \mu(dx)\mu(dy)=\frac{1}{(1-e^{-2\lambda t})^{k/2-1}\sqrt{(1+(k-1)e^{-2\lambda t})^2-k^2e^{-2\lambda t}}}
\end{equation}
\end{lem}

\vspace{1\baselineskip}

We thus have immediatly :

\begin{coro}
For every integer $k$,$\|p(t,.,.)\|_{L^k(\mu\otimes\mu)}$ goes to $1$ when $t$ goes to infinity, and for every $K>1$, there exists $t_K$ such that
\begin{equation}
\sup_{t\geqslant t_K} \|p(t,.,.)\|_{L^k(\mu\otimes\mu)} \leqslant K
\end{equation}

\end{coro}

\begin{proof}
Set $I_k(t)=\int_{\mathbb{R}^2} p(t,x,y)^k \mu(dx)\mu(dy)$.\\

Then, letting $K_t  =  \frac{\lambda}{\sigma^2(1-e^{-2\lambda t})}$, $c_t  = 1+(k-1)e^{-2\lambda t}$, and $d_t  =  ke^{-\lambda t}$, \\

 $\begin{array}{lll}
I_k(t) & = & \frac{\lambda}{\pi\sigma^2}(1-e^{-2\lambda t})^{-k/2}\int_{\mathbb{R}^2}\exp\left(-\frac{\lambda}{\sigma^2(1-e^{-2\lambda t})}\left((1+(k-1)e^{-2\lambda t})(x^2+y^2)-2kxye^{-\lambda t}\right)\right)\;dxdy \\

& = & \frac{\lambda}{\pi\sigma^2}(1-e^{-2\lambda t})^{-k/2}\int_{\mathbb{R}^2}\exp\left(-K_tc_t(x-\frac{d_t}{c_t}y)^2\right)\exp\left(-K_tc_t(1-\frac{d_t^2}{c_t^2})y^2\right)\;dxdy \\

& = & \frac{\lambda}{\pi\sigma^2}(1-e^{-2\lambda t})^{-k/2} \sqrt{\frac{\pi}{K_tc_t}} \sqrt{\frac{\pi}{K_tc_t(1-\frac{d_t^2}{c_t^2})}} \\

& = & \frac{\lambda}{\sigma^2}(1-e^{-2\lambda t})^{-k/2} \frac{1}{K_t \sqrt{c_t^2-d_t^2}} = (1-e^{-2\lambda t})^{1-k/2}\frac{1}{\sqrt{c_t^2-d_t^2}}

\end{array}$

\vspace{1\baselineskip}

Hence the result we were looking for.\\
\end{proof} 

In particular, \begin{equation*}
\|p(a,., .)\|_{\mathcal{L}^8(\mu\otimes\mu)}^8=(1-e^{-2\lambda a})^{-3}(1-50 e^{-2\lambda a}+49 e^{-4\lambda a})^{-1/2}
\end{equation*}
 which is finite if and only $a>\frac{\ln(7)}{\lambda}$.\\
 
 Besides, a study of the function $a \mapsto (1-e^{-2\lambda a})^{-3}(1-50 e^{-2\lambda a}+49 e^{-4\lambda a})^{-1/2}$ shows that it is decreasing towards $1$ on the open interval $\left(\frac{\ln(7)}{\lambda}, +\infty \right)$.\\

 Thus, for every $\delta >\frac{\ln(7)}{\lambda}$, $$\sup_{a\geqslant\delta}\|p(a,., .)\|_{\mathcal{L}^8(\mu\otimes\mu)} <\infty$$ and, furthermore,  \begin{equation}\sup_{a\geqslant\delta}\|p(a,., .)\|_{\mathcal{L}^8(\mu\otimes\mu)}=(1-e^{-2\lambda \delta})^{-3/8}(1-50 e^{-2\lambda \delta}+49 e^{-4\lambda \delta})^{-1/16}=M_{\delta} 
 \end{equation} where $M_{\delta}$ corresponds to the constant defined in (\ref{M}). Its graph can be seen in figure \ref{figure1}, when $\lambda$ is $1$ (as $M_{\delta}$ can be seen as a function of the product $\lambda\delta$, the value chosen for $\lambda$ here is not of much consequence). Notice that it quickly becomes very close to $1$.\\


The perturbed equation is 
\begin{equation}\label{46}
dx_t=\left(-\lambda x_t + \sigma b((x)_{t-t_0}^t)\right)dt + \sigma \; dB_t
\end{equation}
where $b:\mathcal{C}^0([-t_0,0],\mathbb{R})\rightarrow\mathbb{R}$ is a bounded measurable function.

\begin{rem}
Having obtained a bound for $\|b\|_{\infty}$, we can nevertheless make $\sigma$ vary as it does not appear in the various computations in order to obtain a larger drift (though, obviously, it also entails a larger diffusion coefficient).
\end{rem}

\subsection{Numerical applications}

\subsubsection{Evolution of the bound $b(\varepsilon)$}

As $\ln(7) \sim 1.95$, we can for instance choose $\delta = \frac{2}{\lambda}$.\\
If we do so, then $M_{\delta} =(1-e^{-4})^{-3/8}(1-50 e^{-4}+49 e^{-8})^{-1/16}\simeq 1.16$.\\


The evolution of the bound $b(\varepsilon)$, defined in (\ref{38}), supposing that $\lambda=1$, when $\varepsilon$ evolves between $0$ and $10$ can be seen in figure \ref{figure1}.\\

\begin{figure}
   \begin{minipage}[c]{.49\linewidth}
      \includegraphics[scale=0.17]{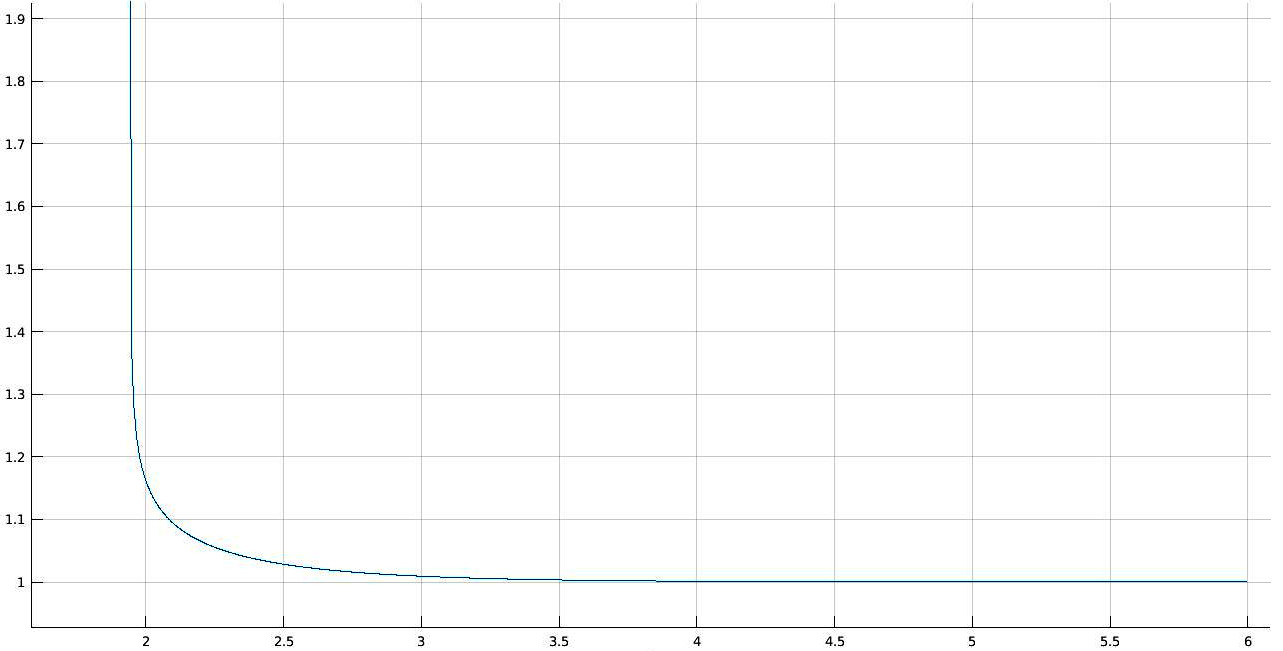}
   \end{minipage}
   \begin{minipage}[c]{.49\linewidth}
      \includegraphics[scale=0.17]{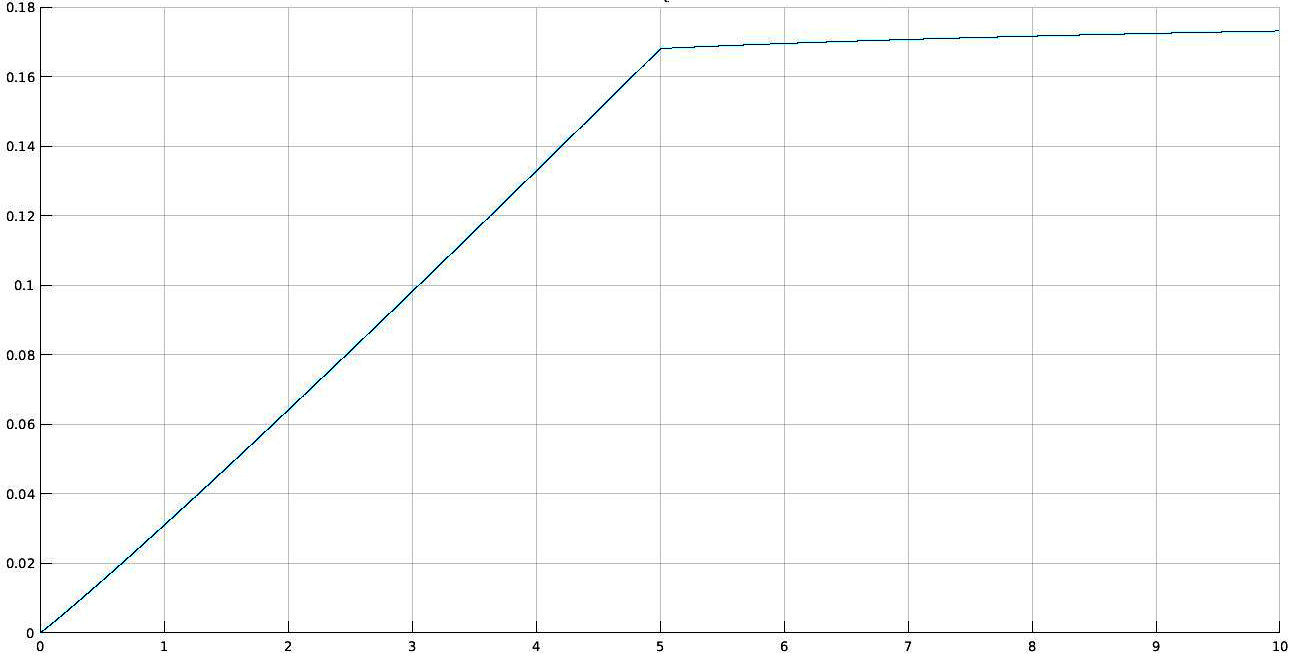}
   \end{minipage}
   \caption{On the left, the graph of $\delta\mapsto M_{\delta}$ when $\lambda=1$ ; on the right, the evolution of $\varepsilon\mapsto b(\varepsilon)$ for $\lambda=1$ and $\delta=2$}\label{figure1}

\end{figure}

This curve is, as expected, non-decreasing (remark that it is not linear by parts) : the smaller $\varepsilon$, the smaller $b(\varepsilon)$ : the result will hold only for very small perturbations $b$ of the reference process.\\

\subsubsection{About the choice of $\delta$}

For now, fix $\lambda=1$. Here, we are interested in the optimal value for $\delta$, that is the $\delta$ which allows the largest possible window of choice for $\|b\|_{\infty}$ in order to satisfy proposition \ref{prop2}.\\

For instance, we set $\varepsilon=1$ and $a=a_1$. Then, considering $b_1:=b(1)$ as a function of $\delta$, we have :
\begin{equation*}
b_1(\delta)=\frac{1}{2\sqrt{2e}(8!)^{1/8}} \left(\delta \;M_{\delta}^{7/4}\left(1 - \frac{C_P}{8 \delta} \ln \left( \frac{2 M_{\delta}^4+1}{16 M_{\delta}^8}\left( \sqrt{1+\frac{1}{32(1+2M_{\delta}^4)^2}}-1 \right) \right)\right) \right)^{-1/2}
\end{equation*}
with $M_{\delta}=(1-e^{-2 \delta})^{-3/8}(1-50 e^{-2 \delta}+49 e^{-4 \delta})^{-1/16}$.\\

We wish to determine, \begin{equation*}
\delta_{*}=\underset{\delta>\ln(7)}{\text{argmax}}\; b_1(\delta)
\end{equation*}
i.e. the value of $\delta$ for which proposition \ref{prop1} will be satisfied for the largest value of $\|b\|_{\infty}$.\\

Differentiating $b_1$ with respect to $\delta$ in order to find the points where the derivative vanishes looks a rather hopeless case.\\

We can nevertheless draw the graph of $b_1$ in figure \ref{figure2} : there seems to be a clear maximum, in this case for $\delta$ close to $2.22$, with a upper-bound for $\|b\|_{\infty}$ around $0.0325$.\\

\begin{figure}
\begin{center}
\includegraphics[scale=0.28]{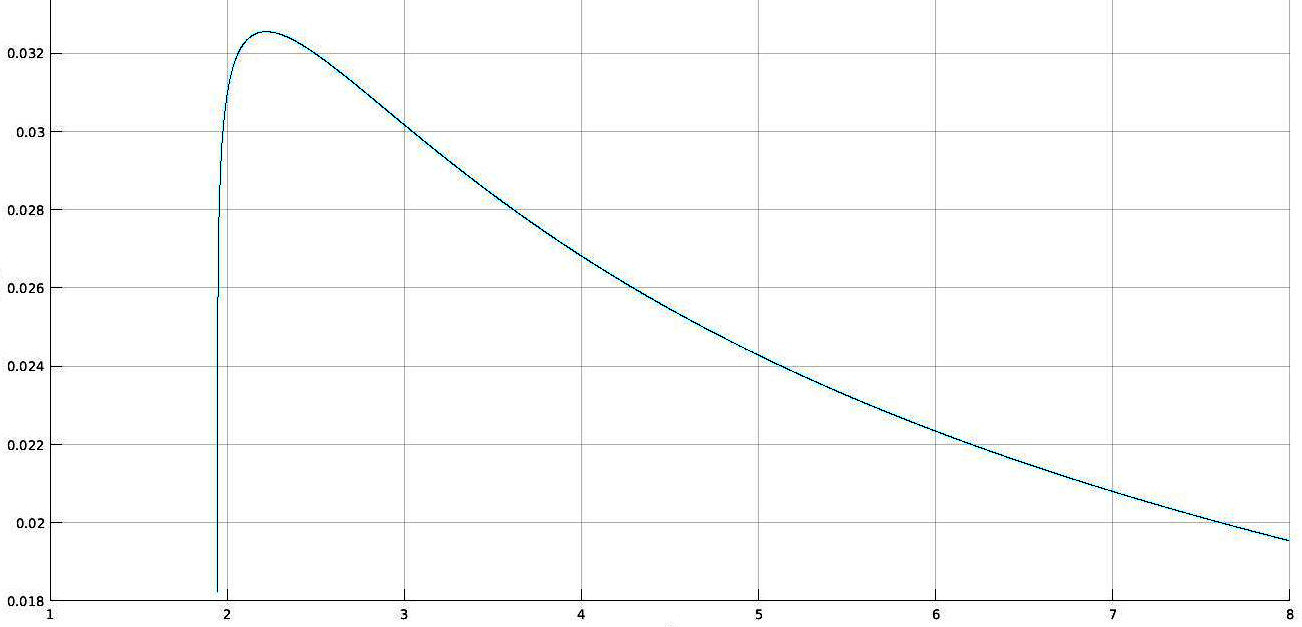}
\caption{$b_1$ as a function of $\delta$ with $\lambda=1$}\label{figure2}
\end{center}
\end{figure}

Other values of $\lambda$ provide similar representations, on a different scale. We observe a surprising correlation, at least at first glance, when looking for the corresponding values for other choices of $\lambda$ ; they are gathered in the following table :\\

\begin{center}
$\begin{array}{|c|c|c|c|c|c|c|c|c|}
\hline 
\lambda & 0.01 & 0.1 & 0.5 & 1 & 2 & 5 & 10 & 100 \\ 
\hline 
\delta_{*} & 222 & 22.2 & 4.44 & 2.22 & 1.11 & 0.44 & 0.222 & 0.0222 \\ 
\hline 
b_1(\delta_{*}) & 0.00325 & 0.0103 & 0.0230 & 0.0325 & 0.0460 & 0.0728 & 0.103 & 0.325 \\ 
\hline 
\end{array}$
\end{center}

From this, we can conjecture approximate links between $\lambda$, $\delta_{*}$ and $b_1(\delta_{*})$ : it seems that, one the one hand,
\begin{equation}
\delta_{*}\simeq \frac{2.220296}{\lambda}
\end{equation}

\begin{figure}
   \begin{minipage}[c]{.49\linewidth}
      \includegraphics[scale=0.17]{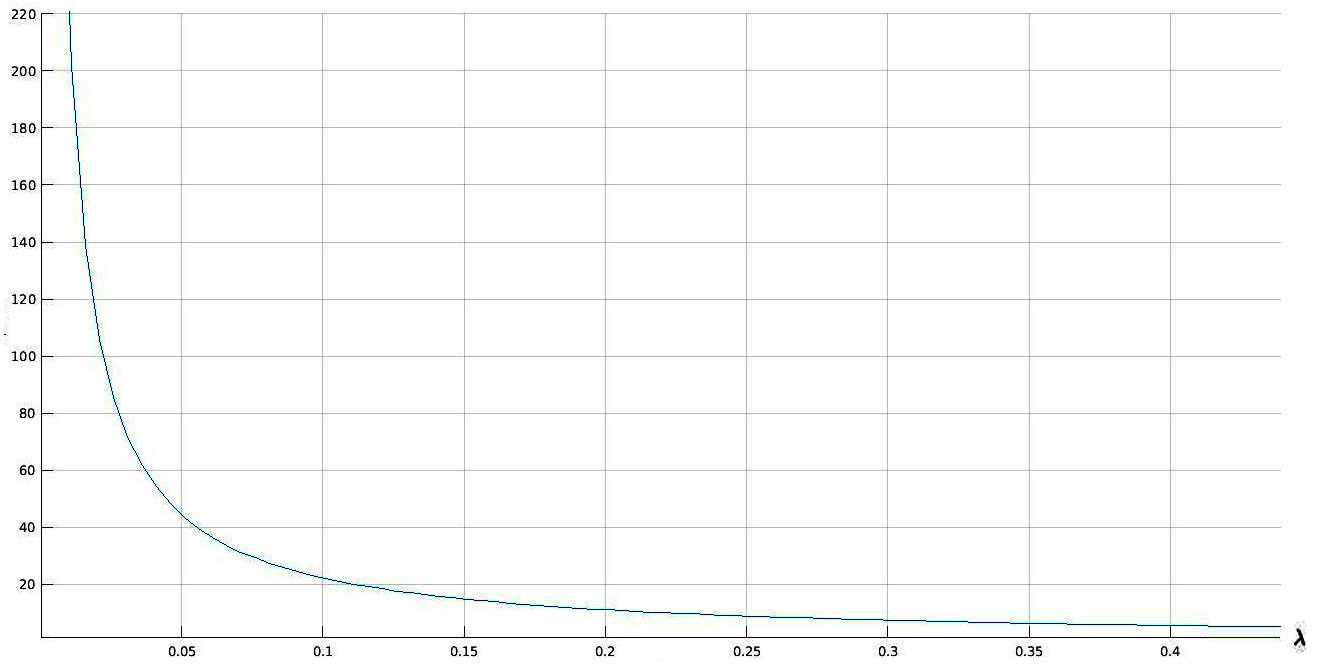}
   \end{minipage}
   \begin{minipage}[c]{.49\linewidth}
      \includegraphics[scale=0.17]{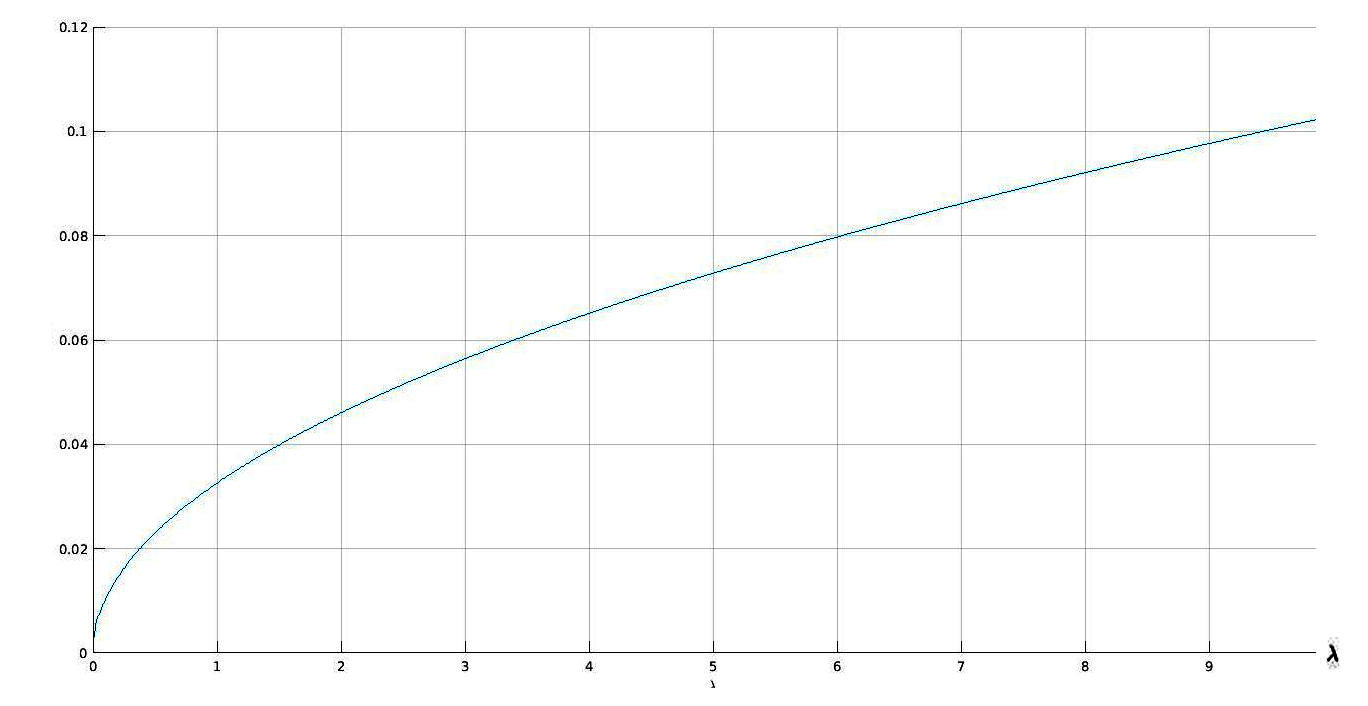}
   \end{minipage}
   \caption{On the left, the optimal value of $\delta$, $\delta_{*}$, as a function of $\lambda$; on the right, the optimal bound $b_1(\delta_{*})$, depending on $\lambda$}\label{figure3}

\end{figure}


and, on the other hand,
\begin{equation}
b_1(\delta_{*})\simeq 0.3255108 \sqrt{\lambda} \simeq \frac{0,04850326}{\sqrt{\delta_{*}}}
\end{equation}

Thus, we conjecture that the larger $\lambda$ (and thus the reference drift term), the smaller $\delta$ and the larger $b_1(\delta_{*})$ (and thus a larger window of possibilities for the perturbation $b$). 


\begin{rem}
The so-called optimal bound for $\|b\|_{\infty}$, that is $b_1(\delta_{*})$, is the one discusses in Proposition \ref{prop1}, and is not linked with the one, unknown, that appears in Theorem \ref{th1}, about which, as previously mentioned, we are not able to say anything.
\end{rem}

\emph{Acknowledgements :} I would like to thank Sylvie Roelly for her sustained help during the writing of this paper.

\bibliographystyle{plain}
\bibliography{biblio}
\end{document}